\documentclass[12pt]{amsart}

\newcommand{\du}[2]{
\ensuremath{	
\begin{matrix}
	\fbox{#1}&    \\
	&   \fbox{#2}
\end{matrix}
}}

\newcommand{\sx}{\secx}
\newcommand{\secx}{\sec x}
\newcommand{\tx}{\tanx}

\newcommand{\dsum}{\di\sum}

\newcommand{\down}{\downarrow}

\DeclareFontFamily{U}{mathx}{\hyphenchar\font45}
\DeclareFontShape{U}{mathx}{m}{n}{
      <5> <6> <7> <8> <9> <10>
      <10.95> <12> <14.4> <17.28> <20.74> <24.88>
      mathx10
      }{}
\DeclareSymbolFont{mathx}{U}{mathx}{m}{n}
\DeclareFontSubstitution{U}{mathx}{m}{n}
\DeclareMathAccent{\widecheck}{0}{mathx}{"71}

\renewcommand{\kill}[1]{}
\newcommand{\dummy}[1]{\mbox{}}

\makeatletter
\newcommand{\xequal}[2][]{\ext@arrow 0055{\equalfill@}{#1}{#2}}
\def\equalfill@{\arrowfill@\Relbar\Relbar\Relbar}
\makeatother

\newcommand{\mto}{\mapsto}

\newcommand{\set}[1]{\ensuremath{\left\{{#1}\right\}}}

\renewcommand{\k}{\ensuremath{\ol{\mathrm{P}}}}

\newcommand{\phx}{\ph{x}}

\newcommand{\h}{\hline}

\newcommand{\si}[1]{\ensuremath{\sin\left({#1}\right)}}

\renewcommand{\k}[1]{\ensuremath{\left({#1}\right)}}

\newcommand{\ds}{\dots}

\newcommand{\bca}{\begin{cases}}
\newcommand{\eca}{\end{cases}}

\newcommand{\mug}{\ensuremath{\infty}}

\newcommand{\tanx}{\ensuremath{\tan x}}

\newcommand{\ff}[2]{\ensuremath{\di\fr{#1}{#2}}}

\newcommand{\s}[1]{\ensuremath{\di\int{#1}\,dx}}

\newcommand{\bpic}{\begin{picture}}\newcommand{\epic}{\end{picture}}

\newcommand{\beda}{\begin{edaenumerate}}
\newcommand{\eeda}{\end{edaenumerate}}

\newcommand{\cd}{\cdots}

\newcommand{\st}{\strut}

\newcommand{\q}{\quad}

\newcommand{\up}{\uparrow}

\newcommand{\bq}{\begin{quote}}\newcommand{\eq}{\end{quote}}

\newcommand{\sig}{\sigma}

\newcommand{\ti}{\times}

\newcommand{\vi}{\\[.1in]}
\newcommand{\viii}{\\[.3in]}
\newcommand{\be}{\begin{enumerate}}\newcommand{\ee}{\end{enumerate}}
\newcommand{\bce}{\begin{center}}\newcommand{\ece}{\end{center}}
\newcommand{\bde}{\begin{description}}\newcommand{\ede}{\end{description}}
\newcommand{\bri}{\begin{flushright}}\newcommand{\eri}{\end{flushright}}
\newcommand{\bb}{\begin{block}}\newcommand{\eb}{\end{block}}
\newcommand{\bt}{\begin{thm}}\newcommand{\et}{\end{thm}}
\newcommand{\bpf}{\begin{proof}}\newcommand{\epf}{\end{proof}}
\newcommand{\bex}{\begin{ex}}\newcommand{\eex}{\end{ex}}
\newcommand{\bexr}{\begin{exr}}\newcommand{\eexr}{\end{exr}}
\newcommand{\bft}{\begin{fact}}\newcommand{\eft}{\end{fact}}
\newcommand{\brk}{\begin{rmk}}\newcommand{\erk}{\end{rmk}}
\newcommand{\ba}{\begin{align*}}\newcommand{\ea}{\end{align*}}
\newcommand{\bexe}{\begin{exe}}\newcommand{\eexe}{\end{exe}}

\newcommand{\bit}{\begin{itemize}}\newcommand{\eit}{\end{itemize}}

\newcommand{\bcm}{}

\newcommand{\hf}{\hfill}
\newcommand{\fr}{\frac}

\newcommand{\bd}{\begin{defn}}\newcommand{\ed}{\end{defn}}
\newcommand{\bp}{\begin{prop}}\newcommand{\ep}{\end{prop}}

\newcommand{\eh}{\emph}

\newcommand{\fb}{\fbox}
\newcommand{\mb}{\mbox}
\newcommand{\te}{\text}\newcommand{\ph}{\phantom}

\newcommand{\lef}{\left}\newcommand{\ri}{\right}

\newcommand{\then}{\Longrightarrow}

\newcommand{\di}{\displaystyle}

\newcommand{\np}{\newpage}

\renewcommand{\up}{\uparrow}

\renewcommand{\int}{\in T}

\renewcommand{\s}{\sigma}

\newcommand{\F}{\mathcal{F}}

\usepackage[dvipdfmx]{graphicx}
\usepackage[dvipsnames]{xcolor}
\usepackage{float,afterpage}
\usepackage{asymptote,layout}
\usepackage{wrapfig,epic}

\usepackage{geometry}
\usepackage{exscale,latexsym,bm}
\usepackage{amssymb,enumerate,amsmath,amsthm,amsfonts}
\usepackage{verbatim,fancybox,wasysym,fancyhdr,type1cm}
\usepackage[frame,all,poly,curve,knot,arrow]{xy}
\usepackage{colortbl}
\usepackage{boxedminipage}
\usepackage{multirow}

\graphicspath{{./figs/}}
\theoremstyle{definition}
\newtheorem{thm}{Theorem}[section]
\newtheorem{lem}[thm]{Lemma}
\newtheorem{prop}[thm]{Proposition}

\newtheorem{exr}[thm]{Exercise}
\newtheorem{ob}[thm]{Observation}

\newtheorem{ex}[thm]{Example}

\newtheorem{defn}[thm]{Definition}\newtheorem{rmk}[thm]{Remark}
\newtheorem{fact}[thm]{Fact}
\newtheorem{block}[thm]{}
\newtheorem*{exe}{Exercise}

\renewcommand{\s}{\sig}
\renewcommand{\phx}{\ph{X}}

\title[]{A new refinement of Euler numbers on counting 
alternating permutations
}
\author[Masato Kobayashi]{Masato Kobayashi$^{*}$}
\date{\today}
\address{Department of Engineering\\
Kanagawa University, 3-27-1 Rokkaku-bashi, Yokohama 221-8686, Japan.}

\thanks{*Department of Engineering, Kanagawa University, Japan}

\renewcommand{\F}{\ensuremath{\mathcal{F}}}

\begin{document}
\renewcommand{\s}{\sigma}

%
%
%
%
%
%
%
%
%
%
%
%
\np

\begin{abstract} 
We often encounter surprising interactions with 
two topics from seemingly different areas.
At a crossroads of calculus and combinatorics, 
the generating function of secant and tangent numbers (Euler numbers)
provides enumeration of alternating permutations. 
In this article, we present a new refinement of Euler numbers to answer the combinatorial question on some particular relation of 
Euler numbers proved by Heneghan-Petersen, Power series for up-down min-max permutations, College Math. Journal, Vol. 45, No. 2 (2014), 83--91.
\end{abstract}
\maketitle
\tableofcontents
\fboxsep4pt

\newcommand{\eup}{E^{\rotatebox{0}{$\uparrow$}}}
\newcommand{\edown}{E^{\rotatebox{0}{$\downarrow$}}}

\newcommand{\ene}{E^{\nearrow}}
\newcommand{\enw}{E^{\nwarrow}}
\newcommand{\ena}{\ene}
\newcommand{\enb}{\enw}


\section{Introduction}

\subsection{calculus, combinatorics, power series}

We often encounter surprising interactions with 
two topics from seemingly different areas.
One such example is a connection between power series in 
calculus and enumeration in combinatorics.
On calculus side, we study \eh{Maclaurin series};
they are convergent power series in the form $\dsum_{n=0}^{\mug} a_{n}x^{n} $ with $(a_{n})$ a sequence of real numbers. Many real-variable differential functions are expressible in this way. A simple example is a geometric series with $a_{n}=1$ for all $n$:
\[
F(x)=1+x+x^{2}+x^{3}+\cd, \q |x|<1.
\]
On combinatorics side, we study a \eh{formal power series}; it is a formal sum of the form $\dsum_{n=0}^{\mug} a_{n}x^{n} $ with $(a_{n})$ a sequence of real numbers  (often integers). Take, again, this as an example:
\[
F(x)=1+x+x^{2}+x^{3}+\cd
\]
with $a_{n}=1$ for all $n$. 
Formally $F(x)$ is a ``rational function" as follows:
\begin{align*}
	F(x)&=1+x(1+x+x^{2}+\cd)=1+xF(x),
	\\(1-x)F(x)&=1,
	\\F(x)&=\ff{1}{1-x}.
\end{align*}
In this computation, we do not need to worry about convergence (which indeed makes perfect sense mathematically; see Wilf \cite{wilf} for such details).
An important class of such series is \eh{exponential generating functions}:
\[
F(x)=\sum_{n=0}^{\mug}a_{n}\ff{x^{n}}{n!}.
\]
In this article, we revisit the interactions between two particular trigonometric functions ($\secx$, $\tanx$) and enumeration of alternating permutations as a follow-up of Heneghan-Petersen \cite{hp}.

\subsection{alternating permutations, Euler numbers}
A \eh{permutation} of degree $n$ is a bijection $\s:\{1, \ds, n\}\to\{1, \ds, n\}$. By $S_{n}$ we denote the set of all permutations of degree $n$.
Say $\s\in S_{n}$ is 
\eh{up-down} if 
\[
\s(1)<\s(2)>\s(3)<\s(4)>\cd.
\]
Say $\s$ is \eh{down-up} if 
\[
\s(1)>\s(2)<\s(3)>\s(4)<\cd.
\]
Further, $\s$ is \eh{alternating} if it is either up-down or down-up;
 in particular, we understand $\s:\{1\}\to \{1\}, 1\mto 1$ is alternating.
 Let $A_{n}$ denote the set of all alternating permutations of degree $n$. Note that some researchers call precisely one of up-down and down-up permutations ``alternating". As mentioned above, we mean both throughout.

Clearly, there is a one-to-one correspondence $\s\leftrightarrow \s^{*}$ 
between up-down and down-up permutations of degree $n$
 where $\s^{*}(i)=n-\s(i)+1$.
For example, $\s=2514376$ (one-line notation: $\s(1)=2, \s(2)=5, \ds$) while $\s^{*}=6374512$.

Now, let us introduce the sequence $(E_{n})_{n=0}^{\mug}$:
First, we formally set $E_{0}=1$. For $n\ge 1$, let 
\begin{align*}
	E_{n}&=|\set{\te{up-down permutations in $S_{n}$}}|
	\\&=|\set{\te{down-up permutations in $S_{n}$}}|.
\end{align*}
Observe that $|A_{1}|=E_{1}=1$ and $|A_{n}|=2E_{n}$ for $n\ge 2$. 

The following surprising result goes back to Andr\'{e} \cite{andre} in 1879:
\[
\dsum_{n=0}^{\mug}
E_{n}\ff{x^{n}}{n!}=\sx+\tx.
\]
In particular, 
\begin{align*}
	\sx&=
	\dsum_{n=0}^{\mug}E_{2n}\ff{x^{2n}}{(2n)!}
	=1+1\ff{x^{2}}{2!}+5\ff{x^{4}}{4!}+61\ff{x^{6}}{6!}+1385\ff{x^{8}}{8!}+\cd,
	\\\tx&=
	\dsum_{n=0}^{\mug}E_{2n+1}\ff{x^{2n+1}}{(2n+1)!}
	=1x+2\ff{x^{3}}{3!}+16\ff{x^{5}}{5!}+272\ff{x^{7}}{7!}+7936\ff{x^{9}}{9!}+\cd
\end{align*}
are its odd and even part, respectively. 
For this reason, 
positive integers $(E_{2n})_{n=0}^{\mug}$ are called 
\eh{secant numbers} and $(E_{2n+1})_{n=0}^{\mug}$ \eh{tangent numbers}.
Together, we call $(E_{n})_{n=0}^{\mug}$ \eh{Euler numbers}.

{\renewcommand{\arraystretch}{1.75}
\begin{table}
\caption{secant and tangent numbers}
\begin{center}
\mb{}\hf
	\begin{tabular}{cccccccc}\h
	$E_{0}$&	$E_{2}$&$E_{4}$&$E_{6}$&$E_{8}$&$\cd$	\\\h
	1&	1&5&61&1385&$\cd$	\\\h
\end{tabular}
\hf
\begin{tabular}{cccccccc}\h
	$E_{1}$&	$E_{3}$&$E_{5}$&$E_{7}$&$E_{9}$&$\cd$	\\\h
	1&	2&16&272&7936&$\cd$	\\\h
\end{tabular}
\hf
\mb{}
\end{center}
\end{table}
}

\renewcommand{\si}{\sig}


%

\section{min-max and max-min permutations}

How can we understand Euler numbers?
One simple idea is a \eh{refinement} of $E_{n}$; 
Recently, Heneghan-Petersen in 2013 introduced 
the refinement by just two sequences $\left(E_{n}^{\nearrow}\right)_{n\ge 2}$, $\lef(E_{n}^{\nwarrow}\ri)_{n\ge2}$ under the name of \eh{min-max} and \eh{max-min} permutations:
\begin{defn}
Say an alternating permutation $\s\in S_{n}$ $(n\ge2)$ is \eh{min-max} if 
\[
\s^{-1}(1)<\s^{-1}(n).
\]
Say it is \eh{max-min} if 
\[
\s^{-1}(n)<\s^{-1}(1).
\]
More informally, this means in one-line notation 
\begin{align*}
	\s&=\cd 1\cd n\cd \q \te{(min-max)},
	\\\s&=\cd n\cd 1\cd \q \te{(max-min)}.
\end{align*}
Define 
$E_{n}^{\nearrow}$, 
$E_{n}^{\nwarrow}$ 
to be the number of min-max, max-min 
up-down permutations of degree ${n}$, respectively.
\end{defn}
Each alternating permutation $\s\in A_{n}$ $(n\ge 2)$ is either min-max or max-min 
 so that those two numbers give a refinement of $E_{n}$:
\[
E_{n}=E_{n}^{\nearrow}+E_{n}^{\nwarrow}.
\]

Heneghan-Petersen further introduced generating functions
\[
\ene(x)=
\dsum_{n=0}^{\infty}{E_{n+2}^{\nearrow}}\ff{x^{n}}{n!}
, \q 
\enw(x)=
\dsum_{n=0}^{\infty}{E_{n+2}^{\nwarrow}}\ff{x^{n}}{n!}.
\]
They proved \cite[Section 4]{hp} that as a formal power series, 
we have 
\[
\ene(x)=\secx\cdot \secx(\secx+\tx),
\]
\[
\enw(x)=\secx\cdot \tx(\secx+\tx).
\]
Moreover, they showed the following relations:
\[
E_{2n+1}^{\nearrow}-E_{2n+1}^{\nwarrow}=0,\]
\[
E_{2n}^{\nearrow}-E_{2n}^{\nwarrow}=E_{2n-2}.\] 
Although they did prove these equalities, they said at the end:
\begin{quote}
\eh{We are left with a tantalizing combinatorial question:
Is there a bijective explanation for}
\[
E_{2n}^{\nearrow}-
E_{2n}^{\nwarrow}=E_{2n-2}?\]
\end{quote}
Our goal is to answer this question.
For this purpose, we will introduce another refinement of $E_{n}$ and then discuss its details.

{\renewcommand{\arraystretch}{1.75}
\begin{table}
\caption{The refinement of Euler numbers in \cite{hp}}
\begin{center}
	\begin{tabular}{c|cccccccccccc}\h
$n$	&	0&1&2&3&4&5&6&7&8&9&$\cd$	\\\h
$E_{n}$	&1	&1&1&2&5&16&61&272&1385&7936&$\cd$	\\\h\h
$\ene_{n}$	&---&---	&	1&1&3&8&33&136&723&3968&$\cd$\\\h
$\enw_{n}$	&---&---	&0&1&2&8&28&136&662&3968&$\cd$	\\\h
\end{tabular}
\end{center}
\end{table}}


\section{New refinement of Euler numbers}

Let $\s$ be an alternating permutation in $S_{n}$ ($n\ge 2$). 
To emphasize ups and downs of such permutations, 
we introduce special notation, a \eh{boxed two-row expression} as follows:
\[
\s=
\begin{matrix}
\fb{5}&&\fb{7}&&\fb{6}&&\fb{4}\st\\
&\fbox{3}&&\fb{1}&&\fb{2}&
\end{matrix}
\]
for $\s=5371624$. 
Besides, let us prepare little terminology for convenience.
\begin{defn}
By the \eh{upper row} of an up-down permutation $\s$ we mean even positions $\{2i\mid 1\le 2i\le n\}$.
For a down-up permutation, it means 
odd positions $\{2i+1\mid 1\le 2i+1\le n\}$.
\end{defn}

\begin{ob}\label{ob1}
$n$, the largest number, must appear in the upper row of $\s$.
\end{ob}

This is clear. More important is the following:
\begin{ob}\label{ob2}
$n-1$, the second largest number, must appear 
either in the upper row of $\s$, or at the extremal position(s) in the lower row; for $n$ even, it means the $1$st position and for $n$ odd, 1st and $n$-th positions. This is because there does not exist two numbers 
in $\{1, 2, \ds, n\}$ which are strictly greater than $n-1$. 
Note also that $n$ must appear right next to $n-1$ (at the second or $(n-1)$-st position) in all such cases.
\end{ob}
\noindent For $n$ even:
\[
\s=
\begin{matrix}
&\fb{\phx}&\\
\fbox{{$n-1$}}&&\fb{\phx}&
\end{matrix}
\cd\,\,
\begin{matrix}
&\fb{\phx}  \\
\fb{\phx}&	
\end{matrix}
\,\,\then\,\,
\begin{matrix}
&\fb{$n$}&&\\
\fbox{{$n-1$}}&&\fb{\phx}&
\end{matrix}
\cd\,\,
\begin{matrix}
&\fb{\phx}  \\
\fb{\phx}&	
\end{matrix}
\]
For $n$ odd:
\[
\resizebox{\hsize}{!}{
$\s=
\begin{matrix}
&\fb{\phx}&&\fb{\phx}&\\
\fbox{{$n-1$}}&&\fb{\phx}
\end{matrix}
\cd\,\,
\begin{matrix}
\fb{\phx}&\\	
&\fb{\phx}  
\end{matrix}
\,\,\then\,\,
\begin{matrix}
&\fb{$n$}&&\\
\fbox{{$n-1$}}&&\fb{\phx}& 
\end{matrix}
\cd\,\,
\begin{matrix}
\fb{\phx}&\\	
&\fb{\phx}  
\end{matrix}$
}
\]
\[
\resizebox{\hsize}{!}{
$\s=
\begin{matrix}
&\fb{\phx}&&\fb{\phx}&\\
\fbox{\phx}&&\fb{\phx}&
\end{matrix}
\cd\,\,
\begin{matrix}
\fb{\phx}&\\	
&\fb{$n-1$}  
\end{matrix}
\,\,\then \,\,
\begin{matrix}
&\fb{\phx}&&\fb{\phx}&\\
\fbox{\phx}&&\fb{\phx}&
\end{matrix}
\cd\,\,
\begin{matrix}
\fb{$n$}&\\	
&\fb{$n-1$}  
\end{matrix}$
}
\]

\begin{defn}
Let $\s$ be an alternating permutation in $S_{n}$.
Say $\s$ is \eh{2nd-max-upper}
if $n-1$ appears in its upper row; 
$\s$ is \eh{2nd-max-lower} if $n-1$ appears in its lower row.
Now for $n\ge 2$, define
\begin{align*}
	\eup_{n}&=|\{\te{up-down 2nd-max-upper permutations in $S_{n}$}\}|,
	\\\edown_{n}&=|\{\te{up-down 2nd-max-lower permutations in $S_{n}$}\}|,
\end{align*}
\end{defn}
These numbers give another refinement of $E_{n}$:
\[
E_{n}=\eup_{n}+\edown_{n}.
\]

\begin{ob}
\label{ob3}
Up to positions of $n-1$ and $n$, there are two kinds of 2nd-max-upper permutations:
\begin{align*}
	\s&=
	\begin{matrix}
		&   \fb{\ph{X}}   \\
		\fb{\ph{X}}&      
	\end{matrix}
	\,\,\cd\,\,
	\begin{matrix}
		&   \fb{$n-1$}   \\
		\fb{\ph{X}}&     
	\end{matrix}
	\,\,\cd\,\,
	\begin{matrix}
		&   \fb{$n$}   \\
		\fb{\ph{X}}&     
	\end{matrix}
\,\,\cd\,\,
	\vi\s&=
	\begin{matrix}
		&   \fb{\ph{X}}   \\
		\fb{\ph{X}}&      
	\end{matrix}
\,\,\cd\,\,
	\begin{matrix}
		&   \fb{$n$}   \\
		\fb{\ph{X}}&     
	\end{matrix}
\,\,\cd\,\,
	\begin{matrix}
		&   \fb{$n-1$}   \\
		\fb{\ph{X}}&     
	\end{matrix}
\,\,\cd\,\,
\end{align*}
Via the transposition $n-1\leftrightarrow n$, we see that 
numbers of those two kinds of permutations are equal.
Hence, $\eup_{n}$ are all even.
\end{ob}

\begin{lem}
We have $\eup_{2}=0$ and 
%
\[
\eup_{n}=
2\cdot (n-2)!
\sum_{
\substack{i, j, k\ge 0\\
(2i+1)+(2j+1)+k=n-2}
}\ff{E_{2i+1}}{(2i+1)!}
\ff{E_{2j+1}}{(2j+1)!}
\ff{E_{k}}{k!}.
\]
for $n\ge 3$.
\end{lem}

\begin{proof}
Suppose $\s$ is an up-down 2nd-max-upper permutation.
Suppose further, for the moment, $n-1$ appears to the left of $n$. 
Let $(i, j, k)$ be a triple of nonnegative integers such that 
$(2i+1)+(2j+1)+k=n-2$. Now, $\s$ looks like this:
\[
\s=
\underbrace{
\begin{matrix}
	&\fb{\phx}\\
	\fb{\phx}&
\end{matrix}
\cd
\begin{matrix}
	&\fb{\phx}\\
	\fb{\phx}&&\fb{\phx}
\end{matrix}
}_{2i+1}
\begin{matrix}
	\fbox{$n-1$}  \\
\mb{}
\end{matrix}
\underbrace{
\begin{matrix}
	&  \fbox{\ph{X}}\\
	\fbox{\ph{X}}&
\end{matrix}
\cd
\begin{matrix}
	&  \fbox{\ph{X}}\\
	\fbox{\ph{X}}&&\fb{\phx}
\end{matrix}
}_{2j+1}
\begin{matrix}
	\fbox{$n$}\\
	\mb{}
\end{matrix}
\underbrace{\begin{matrix}
	&\fbox{\ph{X}}      \\
	\fbox{\ph{X}}&
\end{matrix}
\cd}_{k}
\]
Each part is alternating itself. For the first part, there are 
$\binom{n-2}{2i+1}E_{2i+1}$ choices.
Similarly, for the second part, there are 
$\binom{(n-2)-(2i+1)}{2j+1}E_{2j+1}$ choices, 
and for the third, $E_{k}$.
Altogether, we have 
\[
\sum_{
\substack{i, j, k\ge 0\\
(2i+1)+(2j+1)+k=n-2}
}
\binom{n-2}{2i+1}E_{2i+1}
\binom{(n-2)-(2i+1)}{2j+1}E_{2j+1}
E_{k}\]
\[
=
(n-2)!
\sum_{
\substack{i, j, k\ge 0\\
(2i+1)+(2j+1)+k=n-2}
}\ff{E_{2i+1}}{(2i+1)!}
\ff{E_{2j+1}}{(2j+1)!}
\ff{E_{k}}{k!}.
\]
Taking $n-1 \leftrightarrow n$ into account, $\eup_{n}$
 is equal to the double of this.
\end{proof}

\begin{ex}
Let $n=8$.
All triples $(i, j, k)$ satisfying $(2i+1)+(2j+1)+k=n-2=6$ 
are 
\[
(1, 1, 4), (1, 3, 2), (1, 5, 0), (3, 1, 2), (3, 3, 0), (5, 1, 0).
\]
Therefore, we have 
\begin{align*}
	\eup_{8}&=
	2\ti 6!
	\k{\ff{E_{1}}{1!}\ff{E_1}{1!}\ff{E_{4}}{4!}+
	\ff{E_{1}}{1!}\ff{E_3}{3!}\ff{E_{2}}{2!}+
	\ff{E_{1}}{1!}\ff{E_5}{5!}\ff{E_{0}}{0!}+
	\ff{E_{3}}{3!}\ff{E_1}{1!}\ff{E_{2}}{2!}+
	\ff{E_{3}}{3!}\ff{E_3}{3!}\ff{E_{0}}{0!}+
	\ff{E_{5}}{5!}\ff{E_1}{1!}\ff{E_{0}}{0!}}
	\\&=2(150+120+120+96+80+96)
	\\&=1324.
\end{align*}
This does not seem so elegant at a glance. 
However, we will see in Section 5 that $\{\eup_{n}\}$ has the nice expression 
in terms of formal power series with $\secx$ and $\tanx$.
\end{ex}
\begin{lem}\label{a1}
For $k\ge 1$, we have 
\[
\begin{cases}
	\edown_{2k}=E_{2k-2},\\
	\edown_{2k+1}=2E_{2k-1}.
\end{cases}	
\]
\end{lem}

\begin{proof}
Suppose $\s$ is an up-down 2nd-max-lower permutation in $S_{n}$.
If $n=2k$, then $n-1$ must appear at the first position together with $n$  right after it:
\[
\s=
\begin{matrix}
	&   \fb{$n$}&   &\fb{\phx}\\
	\fb{$n-1$}&   &\fb{\phx}   
\end{matrix}
\,\,\cd\,\,
\begin{matrix}
 &\fb{\phx}\\
	\fb{\phx}&\\
\end{matrix}
\]
Thus, $\edown_{2k}=E_{2k-2}$.
If $n=2k+1$, then $n-1$ must appear at the first position likewise 
or at the last right after $n$:
\[
\s=
\begin{matrix}
	&   \fb{$n$}&   &\fb{\phx}\\
	\fb{$n-1$}&   &\fb{\phx}   
\end{matrix}
\,\,\cd\,\,
\begin{matrix}
 &\fb{\phx}\\
	\fb{\phx}&\\
\end{matrix}
\]
or
\[
\s=
\begin{matrix}
	&   \fb{\phx}&   &\fb{\phx}\\
	\fb{\phx}&   &\fb{\phx}   
\end{matrix}
\,\,\cd\,\,
\begin{matrix}
 &\fb{$n$}\\
\fb{\phx}&	&\fb{$n-1$}&\\
\end{matrix}
\]
Thus, $\edown_{2k+1}=2E_{2k-1}$.
\end{proof}


{\renewcommand{\arraystretch}{1.75}
\begin{table}
\caption{Another refinement of Euler numbers}
\begin{center}
	\begin{tabular}{c|cccccccccccc}\h
$n$	&	0&	1&2&3&4&5&6&7&8&9&$\cd$\\\h
$E_{n}$	&1&	1&1&2&5&16&61&272&1385&7936&$\cd$	\\\h\h\h
$\ena_{n}$	&---	&---	&1&1&3&8&33&136&723&3968&$\cd$\\\h
$\enb_{n}$	&---	&---	&0&1&2&8&28&136&662&3968&$\cd$\\\h\h\h
$\eup_{n}$	&---	&---&0&0&4&12&56&240&1324&7392&$\cd$	\\\h
$\edown_{n}$	&---	&---&1&2&1&4&5&32&61&544&$\cd$	\\\h\h
\end{tabular}
\end{center}
\end{table}}

{\renewcommand{\arraystretch}{1.75}
\begin{table}
\caption{$E^{\up}_{4}=4, E^{\down}_{4}=1, E^{\up}_{5}=12, E^{\down}_{5}=4.$}
\begin{center}
\mb{}\hf	\begin{tabular}{|c|c|c|ccccc}\h
$n=4$	&	2nd-max-upper&2nd-max-lower	\\\h
\multirow{4}{*}{up-down}&	1324&	\multirow{4}{*}{3412}\\
&	1423&	\\
&	2314&	\\
&	2413&	\\\h
\end{tabular}
\hf\mb{}
\viii
\mb{}\hf\begin{tabular}{|c|c|c|ccccc}\h
$n=5$	&	2nd-max-upper&2nd-max-lower	\\\h
up-down&
$\begin{matrix}
	14253&24153   &34152   \\
	14352&24351   &34251\\
	15243&25143   &35142\\
	15342&25341   &35241\\
\end{matrix}$
	&	
$\begin{matrix}
13254	\\
23154	\\
45132	\\
45231	\\
\end{matrix}$
\\\h	
\end{tabular}
\hf\mb{}
\end{center}
\end{table}

}

\section{Theorem}

\begin{thm}
\[
E_{2n}^{\nearrow}-E_{2n}^{\nwarrow}=E_{2n-2} \q (n\ge 1)
\]\end{thm}
Although Heneghan-Petersen proved this using power series, 
here we give a bijective explanation. Let us put it in this way:
\[
E_{2n}^{\nearrow}=E_{2n}^{\nwarrow}+E_{2n-2}.
\]
\begin{lem}\label{l1}
\[
\eup_{2n}=2\enw_{2n} \text{\q for $n\ge 1$}.
\]
\end{lem}

\begin{proof}
We first confirm that $\eup_{2}=0$ and 
\[
\enw_{2n}=(2n-2)!
\sum_{
\substack{i, j, k\ge 0\\
(2i+1)+2j+(2k+1)=2n-2}
}\ff{E_{2i+1}}{(2i+1)!}
\ff{E_{2j}}{(2j)!}
\ff{E_{2k+1}}{(2k+1)!}.
\]
Suppose $\s$ is a max-min permutation of degree $2n$.
It splits into three parts:
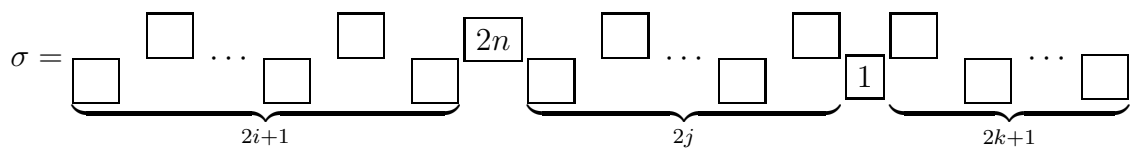
\[
\resizebox{\hsize}{!}{$
\s=
\underbrace{
\begin{matrix}
	&\fb{\phx}\\
	\fb{\phx}&
\end{matrix}
\,\,\cd\,\,
\begin{matrix}
	&\fb{\phx}\\
	\fb{\phx}&&\fb{\phx}
\end{matrix}
}_{2i+1}
\begin{matrix}
	\fbox{$2n$}  \\
\mb{}
\end{matrix}
\underbrace{
\begin{matrix}
	&  \fbox{\ph{X}}\\
	\fbox{\ph{X}}&
\end{matrix}
\,\,\cd\,\,
\begin{matrix}
	&  \fbox{\ph{X}}\\
	\fbox{\ph{X}}&
\end{matrix}
}_{2j}
\begin{matrix}
	\mb{}\\
	\fbox{$1$}
\end{matrix}
\underbrace{\begin{matrix}
	\fbox{\ph{X}}&      \\
	&\fbox{\ph{X}}
\end{matrix}
\,\,\cd\,\,
\begin{matrix}
	 \\
	\fbox{\ph{X}}
\end{matrix}
}_{2k+1}
$}\]
For the first part, there are 
$\binom{2n-2}{2i+1}E_{2i+1}$ choices, 
for the second part, $\binom{(2n-2)-(2i+1)}{2j}E_{2j}$ choices, 
and for the third part, $E_{2k+1}$.
Notice that this enumeration is quite similar to the one for $\eup_{n}$ above; that is, replacing $n$ by $2n$ and $k$ by $2k$ and interchanging $j$ and $k$, 
we see that $\eup_{2n}$ and $2\enw_{2n}$ satisfy exactly the same recurrence with $\eup_{2}=0=2\enw_{2}$. Thus, 
\[
\eup_{2n}=2\enw_{2n} \text{\q for $n\ge 1$}.
\]
It follows from this and Lemma \ref{a1} that 
\[
E_{2n}=\ene_{2n}+\enw_{2n}
=\k{\enw_{2n}+E_{2n-2}}+\enw_{2n}=2\enw_{2n}+E_{2n-2}
=\eup_{2n}+\edown_{2n}.
\]
Hence 
$E_{2n}^{\nearrow}-E_{2n}^{\nwarrow}=E_{2n-2}$ 
 was equivalent to ths idea of our refinement for $E_{2n}$. 
 All the essence is in Observations \ref{ob1}, \ref{ob2} and \ref{ob3}. Very simple!
\end{proof}

\section{Backstage}

Let us understand our result in terms of power series. 
Since $\ene_{n}, \enw_{n}$ make sense for only $n\ge 2$, 
Heneghan-Petersen introduced the power series 
\[
\ene(x)=
\dsum_{n=0}^{\infty}{E_{n+2}^{\nearrow}}\ff{x^{n}}{n!}
, \q 
\enw(x)=
\dsum_{n=0}^{\infty}{E_{n+2}^{\nwarrow}}\ff{x^{n}}{n!}
\]
with two steps shifted. They proved that 
\[
\ene(x)=\sec^{2}x(\sx+\tx), \q 
\enw(x)=\secx\tx(\sx+\tx). \]
As an analogy of this, it is natural to define power series
\begin{align*}
	E^{\uparrow}(x)&=\dsum_{n=0}^{\mug}\eup_{n+2}\ff{x^{n}}{n!}
	=0+0+4\ff{x^{2}}{2!}+12\ff{x^{3}}{3!}+56\ff{x^{4}}{4!}+
	240\ff{x^{5}}{5!}+\cd,
	\\
	E^{\downarrow}(x)&=\dsum_{n=0}^{\mug}\edown_{n+2}\ff{x^{n}}{n!}
=1+2\ff{x}{1!}+1\ff{x^{2}}{2!}+4\ff{x^{3}}{3!}+5\ff{x^{4}}{4!}+
	32\ff{x^{5}}{5!}+\cd.
\end{align*}
It then follows that 
\begin{align*}
E^{\uparrow}(x)
&=\dsum_{n=0}^{\mug}\eup_{n+2}\ff{x^{n}}{n!}
\\
&=
\dsum_{n=0}^{\mug}
2\cdot n!
\k{
\sum_{
\substack{i, j, k\ge 0\\
(2i+1)+(2j+1)+k=n}
}\ff{E_{2i+1}}{(2i+1)!}
\ff{E_{2j+1}}{(2j+1)!}
\ff{E_{k}}{k!}}
\ff{x^{n}}{n!}
\\
&=2
\k{\sum_{i=0}^{\mug}
E_{2i+1}\ff{x^{2i+1}}{(2i+1)!}}
\k{\sum_{j=0}^{\mug}
E_{2j+1}\ff{x^{2j+1}}{(2j+1)!}}
\k{\sum_{k=0}^{\mug}
E_{k}\ff{x^{k}}{k!}}
\\&=2\tx\cdot\tx(\secx+\tx).
\end{align*}
\[
E^{\downarrow}(x)=
\dsum_{n=0}^{\infty}{\edown_{n+2}}\ff{x^{n}}{n!}
=
\dsum_{k=0}^{\infty}{E_{2k}}\ff{x^{2k}}{(2k)!}
+
2\dsum_{k=0}^{\infty}{E_{2k+1}}\ff{x^{2k+1}}{(2k+1)!}
=\secx+2\tx.
\]

After reading \cite{hp}, my question was: ``what do the coefficients of the formal power series 
\[\tanx\cdot \tanx(\sx+\tx) \]
count?" Looking at these three factors thoughtfully, I came up with the idea of 2nd-max-upper permutations and then 
proved subsequent results such as $\eup_{2n}=2\enw_{2n}$
(the even part of $E^{\uparrow}(x)$ is $2\secx\tan^{2}x$ while the one  of $\enw(x)$ is $\secx\tan^{2}x$. Equate the coefficients of $x^{2n-2}/(2n-2)!$ of these power series.)
Using formal power series, it is often easy to predict relations of sequences and give a bijective proof. This is a ``powerful" method.


\section{Final remarks}

This is not the end of the story; any topic in mathematics has its subsequence. We leave two remaining problems on such enumeration.

\begin{itemize}
	\item 
	Heneghan-Petersen proved 
	$\di\lim_{n\to \mug}\ff{\enw_{n}}{\ene_{n}}=1$.
What is 
	$\di\lim_{n\to \mug}\ff{\edown_{n}}{\eup_{n}}$?
	\item Let 
\begin{align*}
D_{n}^{\up}&=|\{\te{down-up 2nd-max-upper permutations in $S_{n}$}\}|,
\\D_{n}^{\down}&=|\{\te{down-up 2nd-max-lower permutations in $S_{n}$}\}|.
\end{align*}
Then $E_{n}=D_{n}^{\up}+D_{n}^{\down}$.
Try computing the first couple of values of $D^{\up}_{n}$, $D^{\down}_{n}$. 
What are exponential generating functions of these sequences?
\end{itemize}




\end{document}